\documentclass[12pt]{amsart}
\usepackage[applemac]{inputenc}
\usepackage{bm}

\usepackage[dvipsnames,svgnames,x11names,hyperref]{xcolor}
\usepackage[hyphens]{url}
\usepackage[backref=page,colorlinks=true,linkcolor=Maroon,citecolor=blue,urlcolor=blue,hypertexnames=false,linktocpage]{hyperref}
\usepackage{bookmark}
\usepackage{amsmath,thmtools,mathtools}
\mathtoolsset{showonlyrefs=true}

\usepackage{blindtext}

\newcounter{mgncount}

\usepackage{tikz}

\usepackage{fancyhdr}
\usepackage{esint}
\usepackage{enumerate}
\usepackage{pictexwd,dcpic}
\usepackage{graphicx}
\usepackage{caption}
\allowdisplaybreaks

\usepackage{graphicx}

\newtheorem*{thm-}{Theorem}
\declaretheorem[name=Theorem,numberwithin=section]{thm}
\declaretheorem[name=Remark,style=remark,sibling=thm]{rem}
\declaretheorem[name=Lemma,sibling=thm]{lemma}
\declaretheorem[name=Proposition,sibling=thm]{prop}

\declaretheorem[name=Corollary,sibling=thm]{cor}

\numberwithin{equation}{section}

\newcommand{\ti}{\widetilde}

\newcommand{\ov}{\bar}

\newcommand{\bbN}{\mathbb{N}}

\newcommand{\bbR}{\mathbb{R}}

\newcommand{\bbS}{\mathbb{S}}

\newcommand{\M}{\mathcal{M}}

\newcommand{\al}{\alpha}

\newcommand{\ga}{\gamma}
\newcommand{\de}{\delta}

\newcommand{\ka}{\kappa}
\newcommand{\la}{\lambda}

\newcommand{\si}{\sigma}

\newcommand{\De}{\Delta}
\newcommand{\Ga}{\Gamma}

\newcommand{\cF}{\mathcal{F}}

\newcommand{\cH}{\mathcal{H}}
\newcommand{\cI}{\mathcal{I}}

\newcommand{\cL}{\mathcal{L}}
\newcommand{\cM}{\mathcal{M}}
\newcommand{\cN}{\mathcal{N}}

\newcommand{\cV}{\mathcal{V}}
\newcommand{\cW}{\mathcal{W}}

\newcommand{\op}[1]{\operatorname{#1}}

\renewcommand{\(}{\left(}
\renewcommand{\)}{\right)}

\newcommand{\pf}[1]{\begin{proof}#1 \end{proof}}
\newcommand{\eq}[1]{\begin{equation}\begin{alignedat}{2} #1 \end{alignedat}\end{equation}}

\newcommand{\ra}{\rightarrow}

\newcommand{\q}{\quad}

\makeatletter
\@namedef{subjclassname@2020}{%
	\textup{2020} Mathematics Subject Classification}
\makeatother

\begin{document}
	\title[New inequalities for non-convex domains]
	{New quermassintegral and Poincar\'{e} type inequalities for non-convex domains}
	\author[Y. Hu, M. N. Ivaki]{Yingxiang Hu, Mohammad N. Ivaki}

\begin{abstract}
In the first part of this paper, we study the following non-homogeneous, locally constrained inverse curvature flow in Euclidean space $\bbR^{n+1}$,
\eq{
\dot{x}=\(\frac{1}{\frac{E_k(\hat\ka)}{E_{k-1}(\hat{\ka})}-\al }-\langle x,\nu\rangle\)\nu, \quad k=2,3,\ldots,n-1. 
}
Assuming that the initial hypersurface $\cM_0 \subset \bbR^{n+1}$ is star-shaped and its shifted principal curvatures $\hat{\ka}=\ka+\al(1,\ldots,1)$ lie in the convex set 
\eq{
\Ga_{\alpha,k}:=\Ga_{k-1}\cap \{\la\in \bbR^n:\, E_k(\la)-\al E_{k-1}(\la)>0\},
}
we show that the flow admits a smooth solution that exists for all positive times, and it converges smoothly to a round sphere. As a corollary, we obtain a new set of Alexandrov-Fenchel-type inequalities for non-convex domains.

In the second part, we derive a Poincar\'{e} type inequality for $k$-convex hypersurfaces which complements a more general version of the well-known Heintze-Karcher inequality.
\end{abstract}

\maketitle	

\section{Introduction}
\subsection{Quermassintegral  inequalities}
The quermassintegrals of a bounded domain $\Omega$ with smooth boundary $\partial \Omega$ in $\bbR^{n+1}$ are defined by 
\eq{
W_0(\Omega)&=\operatorname{Vol}(\Omega),\\ W_k(\Omega)&=\frac{1}{n+1}\int_{\partial\Omega} E_{k-1}(\kappa)d\mu,\quad k=1,\ldots,n+1.
}
Here, $E_i(\kappa)$ is the $i$-th elementary symmetric function of principal curvatures of $\partial\Omega$ normalized so that $E_i(1,\ldots,1)=1$. It is also convenient to define a normalization of quermassintegrals:
\eq{ 
\ti W_k(\Omega)=\frac{1}{n+1-k}W_k(\Omega).
}

The $m$th Garding cone is defined by
\eq{
\Ga_m=\{\la\in \bbR^n :~E_i(\la)>0,\, \forall i=1,\ldots,m\}.
}
For $\al>0$, define
\eq{
\Ga_{\alpha,k}=\Ga_{k-1}\cap \{\la\in \bbR^n:\, E_k(\la)-\al E_{k-1}(\la)>0\}.
}
The set $\Ga_{\alpha,k}$ is convex and contains the shifted cone $\al(1,\ldots,1)+\Gamma_{k}$.

Throughout the paper, all hypersurfaces are considered smooth, closed and embedded.
A hypersurface $\cM\subset \bbR^{n+1}$ is said to be $(\alpha,k)$-convex, if its shifted principal curvatures $\hat{\ka}=\ka+\al(1,\ldots,1)$ at any point belong to $\Ga_{\alpha,k}$.
Moreover, we say $\cM$ is weakly $(\alpha,k)$-convex, if $\hat{\ka}$ lies in the closure of $\Ga_{\alpha,k}$. A hypersurface that is $(0,k)$-convex is said to be $k$-convex.

Fix  $k\in \{2,\ldots, n-1\}$.
Let $\cM_0$ be a $(\alpha,k)$-convex hypersurface. In the first part of this paper, we study the flow of $\cM_0$ under the locally constrained inverse curvature flow
\eq{\label{sum-hessian-flow}
\dot{x}&=\(\frac{1}{F}-\langle x,\nu\rangle\)\nu,\\
F&:=\frac{\sum_{i=0}^{k-1}\beta_iE_{k-i}(\ka)}{\sum_{i=0}^{k-1}\beta_i E_{k-1-i}(\ka)}, \quad \beta_i:=\binom{k-1}{i}\al^i. 
}

\begin{thm}\label{main-thm-flow}
    Let $\cM_0$ be a star-shaped, $(\alpha,k)$-convex hypersurface in $\bbR^{n+1}$. Then the flow \eqref{sum-hessian-flow} admits a star-shaped, $(\alpha,k)$-convex solution $\cM_t$ which exists for all times. Moreover, as $t\ra \infty$, $\cM_t$ converges smoothly to a round sphere with the same value of $\sum_{i=0}^{k-1}\beta_i \ti W_{k-i}$ as $\cM_0$. 
\end{thm}

Let $B_r$ denote a ball of radius $r$. Define
\eq{
g_{\al,k}(r)=\sum_{i=0}^{k-1}\beta_i\ti W_{k-i}(B_r),\q h_{\al,k}(r)=\sum_{i=0}^{k-1}\beta_i\ti W_{k+1-i}(B_r),\q
f_{\al,k}=h_{\al,k}\circ g_{\al,k}^{-1}.
}

\begin{thm}\label{main-thm-ineq}
Suppose $\Omega$ is a bounded domain in $\bbR^{n+1}$ with smooth boundary $\cM=\partial \Omega$. Assume that $\cM$ is star-shaped and $(\alpha,k)$-convex. Then
\eq{\label{new-quermassintegral-ineq}
\sum_{i=0}^{k-1}\beta_i\ti W_{k+1-i}(\Omega)\geq f_{\al,k}\(\sum_{i=0}^{k-1}\beta_i\ti W_{k-i}(\Omega)\),
}
and equality holds only for balls. Moreover, for $k=2$, the inequality holds for star-shaped, weakly $(\alpha,2)$-convex hypersurfaces:
\eq{\label{new-quermassintegral-ineq-xx}
\ti W_{3}(\Omega)+\al \ti W_{2}(\Omega)\geq f_{\al,2}\(\ti W_{2}(\Omega)+\al \ti W_{1}(\Omega)\). 
}
Equality holds only for balls. 
\end{thm}

A vast amount of research on curvature flows is devoted to their applications to various types of geometric inequalities \cite{AHL20,ACW21,BHW16,GLW19,GP17,GW03,GW04,GWW14,HI01,KM14,LG16,LWX14,MS16,Top98,WX14,WZ23,Xia17}. Our main motivation to study the flow \eqref{sum-hessian-flow} involving a quotient in the speed such as 
\eq{
\frac{\sum_{i=0}^{k-1}\beta_iE_{k-i}(\ka)}{\sum_{i=0}^{k-1}\beta_iE_{k-1-i}(\ka)}
}
 primarily stems from the work of Guan and Zhang \cite{GZ21} which considers a class of curvature problems that involves the linear combinations of elementary symmetric functions. Further study of such curvature problems was followed, for example, in \cite{CGH22,Don21,LR23,LRW19,Ren23,Sch08,Sui24,Tsa23}. The design of our flow was inspired by the excellent work of Guan and Li \cite{GL15}. In fact, our flow can be expressed as 
 \eq{
 \dot{x}=\frac{\operatorname{div}(\cV)}{\sum_{i=0}^{k-1}\beta_i E_{k-i}}\nu,\q
 \cV^q:=\frac{1}{2} \left(\sum_{i=0}^{k-1}\beta_i \frac{E_{k-i}}{k-i}\right)^{pq}\nabla_{p}|r|^2.
 }
 The approach of Guan and Li has subsequently been the foundation for a substantial body of work in diverse contexts \cite{CGLS22,CS22,HL22,HL23,HLW22,HWYZ23,HWYZ24,KW23,LS17,LS21,MWW24,SWX22,SX19,Sch22,Sch21,WX22} leading to many beautiful geometric inequalities.

Inequality \eqref{main-thm-ineq} for star-shaped, $(\alpha,k)$-convex hypersurfaces is derived from \autoref{main-thm-flow} in combination with a monotonicity formula (cf. \autoref{prop-monotonicity}). For $k=2$, to prove \eqref{main-thm-ineq} for star-shaped, weakly $(\alpha,2)$-convex hypersurfaces, we employ a modified version of mean curvature flow and a version of Hamilton's maximum principle to smoothly approximate such hypersurfaces by star-shaped, $(\alpha,2)$-convex hypersurfaces; see \autoref{sec: approximation}. 

\begin{rem}\label{rem: conj}
    Let us define
    \eq{
    \ti\Ga_{\al,2}=\Ga_1\cap \{\la\in \bbR^n:\, E_2+\al E_1>0\}.
    }
    Then it is clear that $\Ga_2\subset\ti\Ga_{\al,2}\subset \Ga_1$ for $\al>0$. Moreover, $\cM$ is $(\al,2)$-convex if and only if its principal curvatures lie in the set  $\ti\Ga_{\al,2}$. To see this, note that if $\hat{\ka}\in \Gamma_{\al,2}$, then 
    \eq{
    E_1(\hat{\ka}) &=E_1(\ka)+\al>0, \\
    E_2(\hat{\ka})-\al E_1(\hat{\ka})&=E_2(\ka)+\al E_1(\ka)>0.
    }
    As $E_1(\ka)^2 \geq E_2(\ka)$, we have
    \eq{
    E_1(\ka)(E_1(\ka)+\al)\geq E_2(\ka)+\al E_1(\ka)>0,
    }
    and hence $E_1(\ka)>0$. This shows that $\ka\in \ti\Gamma_{\al,2}$. Conversely, if $\ka\in \ti\Gamma_{\al,2}$, it is clear that $\hat{\ka}\in \Ga_{\al,2}$. Therefore, the inequality \eqref{new-quermassintegral-ineq-xx} holds for hypersurfaces lying in a larger class than the class of 2-convex hypersurfaces. 
    
    We conjecture that the inequality
    \eq{\label{new-quermassintegral-ineq-xxx}
\ti W_{k+1}(\Omega)+\al \ti W_{k}(\Omega)\geq h_{\al,k+1}\circ h_{\al,k}^{-1}\(\ti W_{k}(\Omega)+\al \ti W_{k-1}(\Omega)\)
}
holds for all star-shaped hypersurfaces whose principal curvatures lie in the convex set
   \eq{
    \ti\Ga_{\al,k}:=\Ga_{k-1}\cap \{\la\in \bbR^n:\, E_k+\al E_{k-1}>0\},
    }
and equality holds only for balls. Here, 
\eq{
h_{\al,k}(r):=\ti W_k(B_r)+\al \ti W_{k-1}(B_r)=\frac{\omega_n}{n+1}\left(\frac{r^{n+1-k}}{n+1-k}+\al \frac{r^{n+2-k}}{n+2-k}\right).
}

For $k$-convex hypersurfaces, the following argument verifies \eqref{new-quermassintegral-ineq-xxx}: Suppose $\Omega$ is a smooth bounded domain with star-shaped and $k$-convex boundary (with $2\leq k\leq n-1$). Define
\eq{
\cI_k(\al):=&~\ti W_{k+1}(\Omega)+\al \ti W_{k}(\Omega) -h_{\al,k+1}\circ h_{\al,k}^{-1}\(\ti W_{k}(\Omega)+\al \ti W_{k-1}(\Omega)\).
}
Then inequality \eqref{new-quermassintegral-ineq-xxx} is equivalent to $\cI(\al)\geq 0$. First, we have
\eq{ \label{GL-eq-1}
\cI_k(0) &= \ti W_{k+1}(\Omega)-h_{0,k+1}\circ h_{0,k}^{-1}(\ti W_{k}(\Omega)) \\
       &= \ti W_{k+1}(\Omega)-\ti W_{k+1}(B)\(\frac{\ti W_k(\Omega)}{\ti W_k(B)}\)^\frac{n-k}{n+1-k}.
 }
Note that $W_i(B)=\frac{\omega_n}{n+1}$ for $i=0,1,\ldots,n+1$, where $B$ is the unit ball of $\bbR^{n+1}$. Thus, the inequality $\cI_k(0)\geq 0$ is equivalent to 
\eq{
\(\frac{W_{k+1}(\Omega)}{W_{k+1}(B)}\)^\frac{1}{n+1-(k+1)} \geq \(\frac{W_{k}(\Omega)}{W_{k}(B)}\)^\frac{1}{n+1-k},
}
which is known to hold for star-shaped, weakly $k$-convex hypersurfaces in $\bbR^{n+1}$ by the work of Guan and Li \cite{GL09}. 

Next, we show that $\cI_k'(\al)\geq 0$ for all $\al \geq 0$:
\eq{
\cI_k'(\al)=&~\ti W_k(\Omega)-\left.\frac{\frac{d}{d\al}h_{\al,k+1}}{\frac{d}{d\al}h_{\al,k}}\right|_{r=h_{\al,k}^{-1}(\ti W_k+\al \ti W_{k-1})}\ti W_{k-1}(\Omega)\\
=&~\ti W_k(\Omega)-\frac{n+2-k}{(n+1-k)h_{\al,k}^{-1}\(\ti W_k(\Omega)+\al \ti W_{k-1}(\Omega)\)} \ti W_{k-1}(\Omega).
}
Hence, $\cI_k'(\al)\geq 0$ for $\al \geq 0$ is equivalent to the inequality
\eq{
h_{\al,k}^{-1}\( \ti W_k(\Omega)+\al \ti W_{k-1}(\Omega)\) \geq \frac{n+2-k}{n+1-k}\frac{\ti W_{k-1}(\Omega)}{\ti W_{k}(\Omega)}.
}
By the strict monotonicity of $h_{\al,k}$ in $r$, we obtain
\eq{
\ti W_k(\Omega)+\al \ti W_{k-1}(\Omega) 
\geq &~h_{\al,k}\(\frac{n+2-k}{n+1-k}\frac{\ti W_{k-1}(\Omega)}{\ti W_{k}(\Omega)}\)\\
=&~\frac{\omega_n}{n+1}\frac{W_{k-1}(\Omega)^{n+1-k}}{W_{k}(\Omega)^{n+2-k}}\( \ti W_k(\Omega)+\al \ti W_{k-1}(\Omega)\).
}
Since $\ti W_k(\Omega)+\al \ti W_{k-1}(\Omega)>0$, the above inequality is equivalent to
\eq{ \label{GL-eq-2}
\(\frac{W_{k}(\Omega)}{W_k(B)}\)^\frac{1}{n+1-k} \geq \(\frac{W_{k-1}(\Omega)}{W_{k-1}(B)}\)^\frac{1}{n+1-(k-1)},
}
which holds for star-shaped and weakly $(k-1)$-convex hypersurfaces in $\bbR^{n+1}$.
Thus, we conclude that $\cI_k(0)\geq 0$ and $\cI_k'(\al)\geq 0$ for all $\al\geq 0$. In particular, this implies that $\cI_k(\al)\geq 0$ for all $\al\geq 0$ in the class of star-shaped, weakly $k$-convex hypersurfaces.
\end{rem}

\subsection{A Poincar\'{e} type inequality}
From the standard Reilly's formula (see, for example, \cite{KM18}) it follows that
if $(\cM,g,\nabla)$ (with the induced structure from $\bbR^{n+1}$) is a mean-convex hypersurface, then
\eq{\label{k1-inequality}
\int \frac{(\Delta f)^2}{\cH}  d\mu\geq \int g(\cW(\nabla f),\nabla f) d\mu\quad \forall f\in C^2(\cM),
}
where $\cH$ denotes the mean curvature of $\cM=\partial \Omega$.
Now note that the Heintze-Karcher inequality is a special case of this inequality, which follows by taking $f=\frac{1}{2}|x|^2$. In fact, substituting $f$ into \eqref{k1-inequality} we obtain
\eq{ \label{conj-ineq-f-|x|^2/2}
\int h(x^{\top},x^{\top})\leq &~\int \frac{(n-\cH \langle x,\nu\rangle)^2}{\cH} \\
                 = &~\int \frac{n^2}{\cH}-2n\langle x,\nu\rangle+\cH \langle x,\nu\rangle^2 \\
                 = &~\int \frac{n^2}{\cH}+\cH \langle x,\nu\rangle^2-2n(n+1)\operatorname{Vol}(\Omega),
}
where $x^{\top}$ denotes the tangential component of $x$.
On the other hand, we have
\eq{
\langle x,\nu\rangle\Delta|x|^2=2n \langle x,\nu\rangle-2\cH \langle x,\nu\rangle^2.
}
By integrating by parts, we find
\eq{
\int h(x^{\top},x^{\top})=\int \cH \langle x,\nu\rangle^2-n(n+1)\operatorname{Vol}(\Omega).
}
Hence, the inequality \eqref{conj-ineq-f-|x|^2/2} is equivalent to the Heintze-Karcher inequality:
\eq{
\int \frac{1}{\cH}d\mu\geq \frac{n+1}{n}\operatorname{Vol}(\Omega).
}

In \autoref{sec: poincare-type inequalities}, by establishing an integral identity and employing the concavity of $\sigma_k^{1/k}$ in $\Gamma_k$ we prove the $\sigma_k$ counterpart of the inequality \eqref{k1-inequality}:
\begin{thm}\label{thm: k-inequality}
Let $1< k\leq n$ and $\cM$ be a $k$-convex hypersurface in $\bbR^{n+1}$. Then for all $f\in C^2(\cM)$ we have
\eq{
\int \frac{(\operatorname{tr}_{\dot{\sigma}_k}(\operatorname{Hess}f))^2}{\si_k}d\mu \geq k\int \langle \cW(\nabla f),\nabla f\rangle_{\dot{\sigma}_k}d\mu.
}
\end{thm}

\section{Preliminaries}

\subsection{Notation and basic definitions}
Let $(\bbR^{n+1},\langle~,~\rangle,D)$ be the $(n+1)$-dimensional Euclidean space equipped with its standard inner product and flat connection. The Euclidean space $\bbR^{n+1}$ can be expressed as a warped product manifold $\bbR^{+}\times \bbS^n$ with the metric
\eq{
\ov g=dr^2+r^2 \de,
}
where $r$ is the distance to a fixed point $o\in \bbR^{n+1}$ and $\de$ is the standard metric of the unit sphere $\bbS^{n}$ in $\bbR^{n+1}$. 

Let $\cM$ be a hypersurface in $\bbR^{n+1}$ with unit outward normal $\nu$. Write $g$ for the (standard) induced metric on $\cM$. The second fundamental form $A=(h_{ij})$ of $\cM$ is defined by 
\eq{
h(X,Y)=\langle D_X \nu,Y\rangle
}
for any tangent vectors $X,Y$ on $\cM$.  

Let $\{e_1,\ldots,e_n\}$ be a local orthonormal frame field on $\cM$. Then we write $g_{ij}=g(e_i,e_j)$ and $h_{ij}=h(e_i,e_j)$. The Weingarten map is defined via $\cW=(h_i^j)$, where $h_i^j=g^{kj}h_{ik}$ and $(g^{ij})$ is the inverse matrix of $(g_{ij})$. Moreover, the principal curvatures $\ka=(\ka_1,\ldots,\ka_n)$ of $\cM$ are the eigenvalues of $\cW$. It is convenient to introduce the function $\Phi=r^2/2$. Then, the position vector field of $\cM$ is given by  $x=D \Phi|_{\cM}=r\partial_r|_{\cM}$. Finally, the support function of $\cM$ is defined by 
\eq{
u=\langle x, \nu \rangle.
}

The following formulas hold for smooth hypersurfaces in $\bbR^{n+1}$; see, e.g. \cite[Lem. 2.2 \& 2.6]{GL15}.

\begin{lemma}
    Let $\cM\subset \bbR^{n+1}$ be a hypersurface with the induced metric $g$. Then we have the following formulas:
    \begin{enumerate}
    \item The function $\Phi|_{\cM}$ satisfies
    \eq{ \label{identity-1}
    \nabla_i \Phi&=\langle x,e_i\rangle, \\ \nabla_i \nabla_j \Phi &= g_{ij}-u h_{ij}.
    }
    \item The support function $u$ satisfies
    \eq{ \label{identity-2}
    \nabla_i u&=g^{kl}h_{ik}\nabla_l \Phi,\\
    \nabla_i\nabla_j u&=\langle x,\nabla h_{ij}\rangle+h_{ij}-u(h^2)_{ij},
    }
    where $(h^2)_{ij}=h_{ik}g^{kl}h_{lj}$.
    \end{enumerate}
\end{lemma}

Let $\cF$ denote a symmetric, smooth function of the principal curvatures of a hypersurface:
\[
\cF = \cF(\kappa) = \cF(\kappa_1, \ldots, \kappa_n)=\cF(h_i^j)=\cF(g,h).
\]
We write
\[
\cF_i^j :=\frac{\partial F}{\partial h_j^i},\q \dot{\cF}:=[\cF_i^j]_{1\leq i,j\leq n},\q \cF^{ij} := \frac{\partial \cF}{\partial h_{ij}}, \q \cF^{ij,kl} := \frac{\partial^2 \cF}{\partial h_{ij}\partial h_{kl}}.
\]
For more details on curvature functions, refer to \cite[Chap. 2]{Ger06} and \cite{Sch18}.

\subsection{Elementary symmetric functions}
For any integer $m=1,\ldots,n$, the $m$th elementary symmetric function $\si_m$ is defined by
\eq{
\si_m(\la)=\sum_{1\leq i_1<\cdots<i_m\leq n} \la_{i_1} \cdots \la_{i_m}, \quad \la=(\la_1,\ldots,\la_n)\in \bbR^n.
}
Set $\si_0(\la)=1$ and $\si_m(\la)=0$ for $m>n$. We also define the $m$th normalized elementary symmetric function by 
\eq{
E_m(\la)=\binom{n}{m}^{-1}\si_m(\la), \quad m=1,\ldots,n,
}
and $E_0(\la)=1$, $E_m(\la)=0$ for $m>n$. This definition can be extended to symmetric matrices. Let $A\in \operatorname{Sym}(n)$ be an $n\times n$ symmetric matrix and let $\la=\la(A)$ be the eigenvalues of $A$. Then
\eq{
E_m(A)=E_m(\lambda(A))=\frac{(n-m)!}{n!}\de_{i_1\cdots i_m}^{j_1\cdots j_m}A_{i_1j_1}\cdots A_{i_m j_m}, \quad m=1,\ldots,n.
}

\begin{lemma}\cite{HLP52,Rei73}\label{lem-Em-property}
Set $E_m^{ii}=\frac{\partial E_m}{\partial \la_{i}}$. We have
\eq{
\sum_i E_m^{ii} &= m E_{m-1},\\
\sum_i E_m^{ii}\la_i &= m E_m,\\
\sum_i E_m^{ii}\la_i^2 &= n E_1E_m-(n-m)E_{m+1}.
}
\end{lemma}

\begin{lemma}[Newton-Maclaurin inequalities I] \label{lemma-Newton-ineq}~
\begin{enumerate}
    \item (\cite{Ros89}) For any $\la\in \bbR^n$ with $n\geq 2$, the following Newton inequality holds
    \eq{
E_m^2(\la) \geq E_{m-1}(\la) E_{m+1}(\la), \quad m=1,\ldots,n-1.
}
Moreover, the inequality is strict unless $\la_1=\cdots=\la_n$ or both sides of the inequality vanish.
\item (\cite[Lem. 2.5]{Gua13}) Let $1\leq \ell<k\leq n$. We have
\eq{ \label{Maclaurin-type-I}
E_{k-1}(\la) E_{\ell}(\la) \geq E_k(\la) E_{\ell-1}(\la)\q \forall \la\in \ov\Ga_k
}
and
\eq{ \label{Maclaurin-type-II}
E_\ell^\frac{1}{\ell}(\la) \geq E_{k}^\frac{1}{k}(\la) \q \forall \la\in \ov\Ga_k.
}
Moreover, equality holds in \eqref{Maclaurin-type-I} or \eqref{Maclaurin-type-II} for some $\la\in \Ga_k$ if and only if $\la=(c,\ldots,c)$ for some constant $c>0$.
\end{enumerate}
\end{lemma}

\begin{lemma}[Newton-Maclaurin inequalities II]\label{lemma-new-Newton-ineq}
Let $\al \in \bbR$ and $n\geq 3$. Then for any $\la\in \bbR^n$ and for each $k=2,\ldots,n-1$, we have
\eq{ \label{new-Newton-type-ineq}
     \(\sum_{i=0}^{k-1}\beta_i E_{k-i}(\la)\)^2 
\geq \(\sum_{i=0}^{k-1}\beta_i E_{k+1-i}(\la)\)\( \sum_{i=0}^{k-1}\beta_i E_{k-1-i}(\la) \).
}
Moreover, if $\la+\al(1,\ldots,1)\in \Ga_k$, then equality holds only if $\la_1=\cdots=\la_n$.
\end{lemma}
\begin{proof} 
Let $\hat{\la}=\la+\al(1,\ldots,1)$. Then we have 
\eq{
E_k(\hat{\la})=\sum_{i=0}^{k} \binom{k}{i}\al^{i} E_{k-i}(\la).
}
Recall that
\eq{
\binom{k}{i}=\binom{k-1}{i}+\binom{k-1}{i-1}, \quad 0\leq i\leq k,
}
where we used the convention that $\binom{k}{j}=0$ if $j>k$ or $j<0$. 
By the Newton inequality \eqref{Maclaurin-type-I}, for all $\hat{\la}\in \bbR^n$, there holds
\eq{\label{Newton-type-k}
E_k^2(\hat{\la}) \geq E_{k-1}(\hat{\la})E_{k+1}(\hat{\la}), \quad k=1,2,\ldots,n-1.
}
We have
\eq{
\operatorname{LHS} \text{of \eqref{Newton-type-k}}=&~\bigg(\sum_{i=0}^{k} \binom{k}{i}\al^{i} E_{k-i}(\la)\bigg)^2\\
=&~\bigg(\sum_{i=0}^{k} \(\binom{k-1}{i}+\binom{k-1}{i-1}\) \al^i E_{k-i}(\la)\bigg)^2\\
=&~\bigg(\sum_{i=0}^{k-1} \binom{k-1}{i}\al^{i}E_{k-i}(\la)+\al \sum_{i=0}^{k-1}\binom{k-1}{i} \al^{i} E_{k-1-i}(\la)\bigg)^2 \\
=&~\bigg(\sum_{i=0}^{k-1}\beta_i E_{k-i}(\la)\bigg)^2+\al^2 \bigg(\sum_{i=0}^{k-1}\beta_i E_{k-1-i}(\la)\bigg)^2\\
 &~+2\sum_{i=0}^{k-1} \beta_iE_{k-i}(\la)\sum_{i=0}^{k-1}\beta_i E_{k-1-i}(\la)
}
and
\eq{
 &~\sum_{i=0}^{k+1}\binom{k+1}{i}\al^{i}E_{k+1-i}(\la) \\
=&~\sum_{i=0}^{k} \binom{k}{i}\al^i E_{k+1-i}(\la)+\al \sum_{i=0}^{k} \binom{k}{i}\al^{i}E_{k-i}(\la)\\
=&~\sum_{i=0}^{k-1} \binom{k-1}{i}\al^i E_{k+1-i}(\la)+\al \sum_{i=1}^{k}\binom{k-1}{i-1}\al^{i-1}E_{k+1-i}(\la)\\
 &~+\al \bigg(\sum_{i=0}^{k-1} \binom{k-1}{i}\al^{i}E_{k-i}(\la)+\al\sum_{i=1}^{k} \binom{k-1}{i-1}\al^{i-1}E_{k-i}(\la)\bigg)\\
=&~\sum_{i=0}^{k-1} \beta_i E_{k+1-i}(\la)+2\al \beta_i E_{k-i}(\la)+\al^2 \beta_i E_{k-1-i}(\la).
}
Hence, we obtain
\eq{
\operatorname{RHS} \text{of \eqref{Newton-type-k}}=&~\sum_{i=0}^{k-1}\beta_iE_{k-1-i}(\la) \sum_{i=0}^{k+1}\binom{k+1}{i}\al^{i}E_{k+1-i}(\la)  \\
=&~\sum_{i=0}^{k-1} \beta_i E_{k+1-i}(\la)\sum_{i=0}^{k-1}\beta_iE_{k-1-i}(\la) +\al^2\bigg(\sum_{i=0}^{k-1}\beta_i E_{k-1-i}(\la) \bigg)^2\\
&~+2\al\sum_{i=0}^{k-1}\beta_iE_{k-i}(\la)\sum_{i=0}^{k-1}\beta_i E_{k-1-i}(\la).
}
By subtracting the common terms, we obtain the inequality \eqref{new-Newton-type-ineq}.

If equality holds in $\eqref{new-Newton-type-ineq}$, then equality also holds in \eqref{Newton-type-k}. Then either
\eq{
\hat{\la}_1=\cdots=\hat{\la}_n,
}
or both sides of \eqref{Newton-type-k} vanish. The former case implies that $\la_1=\cdots=\la_n$, while the latter case implies that $E_k(\hat{\la})=0$ (which does not happen due to our assumption).
\end{proof}

We need the following Minkowski integral identity; see, for example, \cite[Prop. 2.5]{GL15}.

\begin{lemma}
    Let $\cM$ be a hypersurface in $\bbR^{n+1}$. Then 
    \eq{ \label{Minkowski-formula}
    \int E_m(\ka) d\mu =\int uE_{m+1}(\ka)d\mu, \quad m=0,1,\ldots,n-1,
    }
    where $E_m(\ka)$ is the normalized $m$th mean curvature of $\cM$. 
\end{lemma}

\subsection{Properties of a curvature function}
Recall that
\eq{
F(\la)=\frac{E_k(\hat\la)}{E_{k-1}(\hat\la)}-\al=\frac{\sum_{i=0}^{k-1}\beta_iE_{k-i}(\la)}{\sum_{i=0}^{k-1}\beta_iE_{k-1-i}(\la)}, \quad k=2,\ldots, n-1.
}
It is clear that in $\Gamma_{\al,k}$, $F>0$, $F^{ii}>0$ and $F$ is concave.

\begin{lemma}\label{boundary of Ga-al-k-cone}
We have
\eq{\partial \Ga_{\alpha, k}=\bar{\Ga}_{k-1}\cap \{\la\in \bbR^n:\, E_k(\la)=\al E_{k-1}(\la)\}.
}
In particular, $F|_{\partial\Ga_{\alpha,k}}=0$.
\end{lemma}
\pf{
Take $\la\in \partial \Ga_{\alpha, k}$. If $E_k(\la)>\al E_{k-1}(\la)$, then we would have $E_{k-1}(\la)=0$ and $E_k(\la)>0$. Therefore, $\la\in \ov\Gamma_k$. Now, due to the Maclaurin inequality \eqref{Maclaurin-type-II}, $E_{k-1}(\la) \geq E_{k}^\frac{k-1}{k}(\la)>0$, which violates $E_{k-1}(\la)=0$.
}

\begin{lemma}\label{lem-algebraic-property-F}
We have
\eq{
\sum_{i}F^{ii}\la_i\leq F,\quad \sum_i F^{ii}\la_i^2 \geq F^2, \quad \sum_i F^{ii}\geq 1. 
}
\end{lemma}
\begin{proof}
Recall that $F=\frac{E_k(\hat{\la})}{E_{k-1}(\hat{\la})}-\al$. We calculate
\eq{
\sum_{i}F^{ii}=\sum_i \frac{\partial}{\partial \la_i}\(\frac{E_k(\hat{\la})}{E_{k-1}(\hat{\la})}\)=\sum_i \frac{\partial}{\partial \hat{\la}_i}\(\frac{E_k(\hat{\la})}{E_{k-1}(\hat{\la})}\)\geq 1,
}
and
\eq{
\sum_{i}F^{ii}\la_i
                   &=\sum_{i}F^{ii}\hat{\la}_i-\al \sum_{i}F^{ii}\\
                   &\leq \sum_{i}\hat{\la}_i\frac{\partial}{\partial \hat{\la}_i}\(\frac{E_k(\hat{\la})}{E_{k-1}(\hat{\la})}\)-\al=F,
}
where we used the one-homogeneity of $E_k/E_{k-1}$.

Due to \autoref{lem-Em-property}, we have
\eq{
\sum_i F^{ii}\hat{\la}^2_i &=\sum_i\frac{\partial}{\partial \hat{\la}_i}\(\frac{E_k(\hat{\la})}{E_{k-1}(\hat{\la})}\)\hat{\la}^2_i\\
&=\sum_i \(\frac{E_k^{ii}(\hat{\la})}{E_{k-1}(\hat{\la})}-\frac{E_k(\hat{\la})E_{k-1}^{ii}(\hat{\la})}{E_{k-1}^2(\hat{\la})}\)\ \hat\la_i^2\\
&=(n+1-k)\frac{E_k^2(\hat{\la})}{E_{k-1}^2(\hat{\la})}-(n-k)\frac{E_{k+1}(\hat{\la})}{E_{k-1}(\hat{\la})}\\
&\geq \frac{E_{k}^2(\hat{\la})}{E_{k-1}^2(\hat{\la})}=(F+\al)^2,
}
where the last inequality follows from the Newton inequality 
\eq{
E_{k-1}(\hat{\la})E_{k+1}(\hat{\la})\leq E_{k}(\hat{\la})^2\q \text{and}\q E_{k-1}(\hat{\la})>0.
}
Hence,
\eq{
\sum_{i}F^{ii}\la_i^2 &=\sum_{i}F^{ii}(\hat{\la}_i^2-2\al \hat{\la}_i+\al^2) \\
                      &=\sum_{i}F^{ii} \hat{\la}^2_i-2\al \sum_{i}F^{ii}\hat{\la}_i+\al^2 \sum_{i}F^{ii}\\
                      &\geq (F+\al)^2-2\al (F+\al)+\al^2=F^2.
}

\end{proof}

\section{Regularity estimates}
We start by deriving a few standard evolution equations for the flow 
\eq{\label{general-flow}
\dot{x}=\(\frac{1}{F}- u\)\nu,\quad
F=\frac{E_k(\hat\ka)}{E_{k-1}(\hat\ka)}-\al,\q k=2,\ldots,n-1.
}

\begin{lemma}
Along the flow
$
\dot{x}=\cF \nu,
$
we have 
\begin{align}
\partial_t \nu &=-\nabla \cF, \label{evol-unit-normal-general}\\
\partial_t h_i^j &=-\nabla_i\nabla^j \cF -\cF (h^2)_i^j. \label{evol-Weingarten-matrix-general}
\end{align}
\end{lemma}

\begin{lemma}
We have the following evolution equations along the flow \eqref{general-flow}.
        \eq{\label{evol-hij}
\partial_t h_i^j &= \frac{F^{kl}}{F^2}\nabla_k\nabla_l h_i^j+\(\frac{F^{kl}}{F^2}(h^2)_{kl}+1\)h_i^j+\frac{1}{F^2}F^{kl,pq}\nabla_i h_{pq}\nabla^j h_{kl}\\
                &\quad -\frac{2}{F^3}\nabla_i F\nabla^j F+\langle \nabla h_i^j,x\rangle -\(\frac{F^{kl}}{F^2}h_{kl}+\frac{1}{F}\)(h^2)_i^j. 
}    
        \eq{\label{evol-F}
\partial_t F &=\frac{F^{ij}}{F^2}\nabla_i\nabla_j F-\frac{2F^{ij}}{F^3}\nabla_iF \nabla_j F+\langle \nabla F,x\rangle\\
             &\quad + F^{ij}h_{ij}-\frac{1}{F}F^{ij}(h^2)_{ij}.
}
        \eq{\label{evol-u}
\partial_t u &=\frac{F^{kl}}{F^2}\nabla_k\nabla_l u+\langle x,\nabla u\rangle +\frac{1}{F}\(1-\frac{F^{kl}h_{kl}}{F}\)\\
&\quad +u\(\frac{F^{kl}(h^2)_{kl}}{F^2}-1\).
}
\end{lemma}
\begin{proof}
(i) Using \eqref{evol-Weingarten-matrix-general}, we have
\eq{
\partial_t h_i^j&=-\nabla_i \nabla^j\(\frac{1}{F}-u\)-\(\frac{1}{F}-u\) (h^2)_i^j \\
                &=\frac{1}{F^2}\nabla_i\nabla^j F-\frac{2}{F^3}\nabla_i F\nabla^j F+\nabla_i\nabla^j u-\frac{1}{F}(h^2)_i^j+u(h^2)_i^j\\
                &=\frac{F^{kl}}{F^2}\nabla_k\nabla_l h_i^j+\(\frac{F^{kl}}{F^2}(h^2)_{kl}+1\)h_i^j+\frac{1}{F^2}F^{kl,pq}\nabla_i h_{pq}\nabla^j h_{kl}\\
                &\quad -\frac{2}{F^3}\nabla_i F\nabla^j F+\langle \nabla h_i^j,x\rangle -\(\frac{F^{kl}}{F^2}h_{kl}+\frac{1}{F}\)(h^2)_i^j,
}
where we used  
\eq{
\nabla_i\nabla_j F&=F^{kl}\nabla_i\nabla_j h_{kl}+F^{kl,pq}\nabla_i h_{pq}\nabla_j h_{kl}\\
                  &=F^{kl}\nabla_k\nabla_l h_{ij}+F^{kl}(h^2)_{kl} h_{ij}-F^{kl}h_{kl} (h^2)_{ij}+F^{kl,pq}\nabla_i h_{pq} \nabla_j h_{kl}
}
and
$
\nabla_i\nabla^ju=\langle \nabla h_i^j, x\rangle+h_i^j-u(h^2)_i^j
$ due to \eqref{identity-2}.

(ii) To deduce \eqref{evol-F}, by using \eqref{evol-Weingarten-matrix-general} and \eqref{identity-2} we compute
\eq{
    \partial_t F&=F^i_j\partial_t h_i^j \\
            &= -F^i_j\nabla_i\nabla^j \(\frac{1}{F}-u\)-\(\frac{1}{F}-u\)F^{ij}(h^2)_{ij}\\
            &=F^i_j\( \frac{1}{F^2}\nabla_i \nabla^j F-\frac{2}{F^3}\nabla_i F\nabla^j F\)+F^i_j(\langle \nabla h_i^j, x\rangle+h_i^j-u(h^2)_i^j)\\
            &\quad -\(\frac{1}{F}-u\)F^{ij}(h^2)_{ij}\\
            &=\frac{F^{ij}}{F^2}\nabla_i\nabla_j F-\frac{2F^{ij}}{F^3}\nabla_iF \nabla_j F+\langle \nabla F,x\rangle+ F^{ij}h_{ij}-\frac{F^{ij}}{F}(h^2)_{ij}.
}

(iii)  Owing to $\partial_t\nu=-\nabla(\frac{1}{F}-u)$ (see \eqref{evol-unit-normal-general}), we have
\eq{ \label{evol-u-deduction-1}
\partial_t u&=\frac{1}{F}-u-\langle x,\nabla(\frac{1}{F}-u)\rangle\\
&=\frac{1}{F}-u+\frac{1}{F^2}\langle x,\nabla F\rangle +\langle x,\nabla u\rangle.
}
On the other hand, by \eqref{identity-2}, we also have 
\eq{ \label{evol-u-deduction-2}
\frac{F^{ij}}{F^2}\nabla_i\nabla_j u&=\frac{F^{ij}}{F^2}(\langle \nabla h_{ij},x\rangle+h_{ij}-u(h^2)_{ij})\\
&=\frac{1}{F^2}(\langle \nabla F,x\rangle+F^{ij}h_{ij}-uF^{ij}(h^2)_{ij}).
}
Now \eqref{evol-u} follows immediately from \eqref{evol-u-deduction-1} and \eqref{evol-u-deduction-2}.
\end{proof}

\begin{lemma}\label{curvature est MP}Along the flow \eqref{general-flow}, we have
\eq{
\partial_t \frac{h_i^j}{u}&= \frac{F^{kl}}{F^2}\nabla_k\nabla_l \frac{h_i^j}{u}+\langle \nabla \frac{h_i^j}{u},x\rangle +\frac{2}{uF^2}F^{kl}\nabla_k\frac{h_i^j}{u}\nabla_l u\\
                &\quad +\frac{1}{uF^2}F^{kl,pq}\nabla_i h_{pq}\nabla^j h_{kl}-\frac{2}{uF^3}\nabla_i F\nabla^j F\\
                &\quad -\(\frac{F^{kl}h_{kl}}{F}+1\)\frac{(h^2)_i^j}{uF}-\(\frac{1}{uF}\left(1-\frac{F^{kl}h_{kl}}{F}\right)-2\)\frac{h_i^j}{u}.  
}
\end{lemma}
\begin{proof}
    By \eqref{evol-hij} and \eqref{evol-u}, we have
    \eq{
    \partial_t\frac{h_i^j}{u}&=\frac{1}{u}\partial_t h_i^j-\frac{h_i^j}{u^2}\partial_t u\\
                              &=\frac{F^{kl}}{uF^2}\nabla_k\nabla_l h_i^j+\(\frac{F^{kl}(h^2)_{kl}}{F^2}+1\)\frac{h_i^j}{u}+\frac{1}{uF^2}F^{kl,pq}\nabla_i h_{pq}\nabla^j h_{kl}\\
                &\quad -\frac{2}{uF^3}\nabla_i F\nabla^j F+\langle \nabla \frac{h_i^j}{u},x\rangle -\(\frac{F^{kl}h_{kl}}{F}+1\)\frac{(h^2)_i^j}{uF}\\
                &\quad -\frac{h_i^j}{u^2F^2}F^{kl}\nabla_k\nabla_l u-\frac{h_i^j}{u^2F}\(1-\frac{F^{kl}h_{kl}}{F}\)-\frac{h_i^j}{u}\(\frac{F^{kl}(h^2)_{kl}}{F^2}-1\) \\
                &=\frac{F^{kl}}{F^2}\nabla_k\nabla_l \frac{h_i^j}{u}+\langle \nabla \frac{h_i^j}{u},x\rangle +\frac{2}{F^2}F^{kl}\nabla_k\frac{h_i^j}{u}\nabla_l \log u\\
                &\quad +\frac{1}{uF^2}F^{kl,pq}\nabla_i h_{pq}\nabla^j h_{kl}-\frac{2}{uF^3}\nabla_i F\nabla^j F\\
                &\quad -\(\frac{F^{kl}h_{kl}}{F}+1\)\frac{(h^2)_i^j}{uF}-\(\frac{1}{uF}\left(1-\frac{F^{kl}h_{kl}}{F}\right)-2\)\frac{h_i^j}{u}.
    }
\end{proof}

Let $(\cM,g)$ be a hypersurface in $\bbR^{n+1}$ with the induced metric $g$. We say $\cM$ is star-shaped if the support function is positive everywhere on $\cM$. In this case, $\cM$ can be expressed as a radial graph in spherical coordinates in $\bbR^{n+1}$:
\eq{
\cM=\{ (r(z),z)\in \bbR^{+}\times \bbS^n ~|~z \in \bbS^n \},
}
where $r=r(z)$ is a smooth and positive function over $\bbS^n$, and $z=(z^1,\ldots,z^n)$ are local coordinates of $(\bbS^n,\de,\ov\nabla)$. 

We write $\partial_i=\partial_{z^i}$ and $r_i=\ov\nabla_i r$. The tangent space of $\cM$ is spanned by 
\eq{
\{e_i:=\partial_i+r_i \partial_r,~i=1,\ldots,n\},
}
 and we have
\eq{
u&=\frac{r}{v}, \quad v=\sqrt{1+r^{-2}|\ov\nabla r|^2},\\
g_{ij}&=r^2 \de_{ij}+\bar{\nabla}_ir\bar{\nabla}_jr,\quad g^{ij}=r^{-2}\(\de^{ij}-\frac{\bar{\nabla}^ir\bar{\nabla}^jr}{r^2+|\ov\nabla r|^2}\),\\
h_{ij}&=\frac{1}{rv}(-r\bar{\nabla}_i\bar{\nabla}_jr+2\bar{\nabla}_ir \bar{\nabla}_jr+r^2\de_{ij}),\\
h_i^j&=\frac{1}{r^{3} v}\(\de^{jk}-\frac{\bar{\nabla}^j r \bar{\nabla}^kr}{r^2 v^2}\)\(-r \bar{\nabla}_k\bar{\nabla}_i r+2\bar{\nabla}_ir\bar{\nabla}_kr+r^2\de_{ki}\).
}
Let us introduce a new variable $\ga=\log r$. Then 
\eq{
u&=\frac{e^\ga}{v}, \quad v=\sqrt{1+|\ov\nabla \ga|^2},\\
g_{ij}&=e^{2\ga}(\de_{ij}+\bar{\nabla}_i \ga \bar{\nabla}_j \ga),\quad g^{ij}=e^{-2\ga}\(\de^{ij}-\frac{\bar{\nabla}^i\ga \bar{\nabla}^j\ga}{v^2}\),\\
h_{ij}&=\frac{e^\ga}{v}(-\bar{\nabla}_i\bar{\nabla}_j\ga+\bar{\nabla}_i\ga \bar{\nabla}_j \ga+\de_{ij}),\\
h_i^j&=\frac{1}{e^\ga v}\(\de^{jk}-\frac{\bar{\nabla}^j\ga\bar{\nabla}^k\ga}{v^2}\)\(-\bar{\nabla}_k\bar{\nabla}_i\ga+\bar{\nabla}_k\ga \bar{\nabla}_i\ga+\de_{ki}\).
}

Suppose $\cM_t$, $t\in [0,T)$, is a family of star-shaped, $(\alpha,k)$-convex hypersurfaces solving \eqref{general-flow}. The corresponding $\ga=\ga(x(z,t),t)$ satisfies the following initial value problem:
\eq{ \label{general-flow-radial}
\left\{
\begin{aligned}\partial_t\ga&=\frac{v}{e^\ga F}-1, \quad (z,t)\in \bbS^n \times [0,T),\\
                \ga(\cdot,0)&=\ga_0(\cdot),
\end{aligned}\right.
}
where $\ga_0=\log r_0$ and $r_0$ is the radial function of $\cM_0$.

\begin{lemma}\label{prop-C0-estimate}
   We have
    \eq{ \label{C0-estimate}
    \min_{\bbS^n} \ga(x(\cdot,0),0) \leq \ga(x(z,t),t) \leq \max_{\bbS^n} \ga(x(\cdot,0),0). 
    }
\end{lemma}
\begin{proof}
Assume that $\ga$ attains its maximum at some point $z_0\in \bbS^n$. Then, at $z_0$, we have $\ov\nabla \ga=0$, $v=1$ and $\ov\nabla^2\ga \leq 0$. Therefore,
\eq{
h_i^j  \geq e^{-\ga} \de_i^j.
}
Since $F$ is strictly increasing in $\ka$, we have 
\eq{
F(\ka) \geq F(e^{-\ga} I)=\frac{E_k((e^{-\ga}+\al)I)}{E_{k-1}((e^{-\ga}+\al)I)}-\al=e^{-\ga}.
}
Due to \eqref{general-flow-radial}, the upper bound of $\ga(x(z,t),t)$ follows from
\eq{
\frac{d}{dt}\max_{\bbS^n}\ga(x(\cdot,t),t) \leq 0,
}
at differentiable point of $\max_{\bbS^n}\ga(x(\cdot,t),t) $.
The lower bound of $\ga(x(z,t),t)$ can be deduced similarly. 
\end{proof}

\begin{lemma}\label{prop-C1-estimate}
We have
$
u \geq \min_{\cM_0} u.
$
Moreover, there exists a positive constant $C_1$ depending only on $\cM_0$ such that 
\eq{ \label{C1-estimate}
|\ov\nabla \ga|(z,t)\leq C_1.
}
\end{lemma}
\begin{proof}
    Recall the evolution equation \eqref{evol-u}:
    \eq{
    \partial_t u = \frac{F^{kl}}{F^2}\nabla_k\nabla_l u+\langle x,\nabla u\rangle+\frac{1}{F}\(1-\frac{F^{kl}h_{kl}}{F}\)+u\(\frac{F^{kl}(h^2)_{kl}}{F^2}-1\).
    }
    Due to \autoref{lem-algebraic-property-F}, we have
        \eq{
    1-\frac{F^{kl}h_{kl}}{F}\geq 0, \quad \frac{F^{kl}(h^2)_{kl}}{F^2}-1\geq 0.
    }
    Hence, we obtain
    \eq{
    \frac{d}{dt}\min_{\cM_t}u(\cdot,t) \geq 0.
    }
The maximum principle implies that $u$ is bounded below away from zero and
    \eq{
    \frac{e^\ga}{\sqrt{1+|\ov\nabla \ga|^2}}=u \geq C.
    }
Now the upper bound on $|\ov\nabla \ga|$ follows from \eqref{C0-estimate}.
\end{proof}

\begin{lemma}\label{lemma-estimate-F}
There exists $C_2>0$ depending only on $\cM_0$ such that 
\eq{
C_2 \leq F \leq \max_{\cM_0} F.
}
\end{lemma}
\begin{proof}
Recall the evolution equation \eqref{evol-F}:
\eq{
\partial_t F &=\frac{F^{ij}}{F^2}\nabla_i\nabla_j F-\frac{2F^{ij}}{F^3}\nabla_iF \nabla_j F+\langle \nabla F,x\rangle+ F^{ij}h_i^j-\frac{1}{F}F^{ij}(h^2)_i^j.
}
By \autoref{lem-algebraic-property-F}, we have
\eq{
F^{ii}\ka_i-\frac{F^{ii}\ka_i^2}{F} \leq 0.
}
Hence, from the maximum principle it follows that $F\leq \max_{\cM_0} F$. 

To obtain the lower bound of $F$, note that by \eqref{general-flow-radial}:
\eq{ \label{eq-partial_t-gamma}
\partial_t \ga = \frac{v}{e^\ga F(\frac{1}{e^\ga v}\ti g^{jk}(-\ga_{;ik}+\ga_{;i}\ga_{;k}+\de_{ki}))}-1=:G(\ov\nabla^2 \ga,\ov\nabla \ga,\ga),
}
where $\ti g^{ij}:=e^{2\ga}g^{ij}=\delta^{ij}-\frac{\gamma^i\gamma^j}{v^2}$. Now, differentiating \eqref{eq-partial_t-gamma} with respect to $t$ gives
\eq{
\partial_t (\partial_t\ga)= G^{ij}(\partial_t \ga)_{;ij}+G^{(\ov\nabla \ga)^p} (\partial_t\ga)_{;p} +G^\ga \partial_t \ga,
}
where $G^\ga:=\frac{\partial G}{\partial \ga}$ and so on.
Due to $F^{ii}\ka_i \leq F$ (cf. \autoref{lem-algebraic-property-F}),
\eq{
G^\ga&=-\frac{v}{e^{\ga}F}+\frac{v}{e^\ga}\frac{F^i_j}{F^2}\bigg(\frac{1}{e^\ga v}\ti g^{jk}(-\ga_{;ik}+\ga_{;i}\ga_{;k}+\de_{ki})\bigg) \\
&=-\frac{v}{e^{\ga}F}+\frac{v}{e^{\ga}F^2}F^{ii}\ka_i \\
&\leq 0.
}
This implies that $\partial_t \ga \leq C$ and 
\eq{
\frac{v}{e^\ga F}-1 \leq C.
}
In view of \autoref{prop-C0-estimate}, we have $F\geq C_2>0$, where the constant $C_2$ depends only on $\cM_0$. 
\end{proof}

\begin{lemma}\label{preserving (a,k)-convexity}
There exists $C_3>0$ depending only on $\cM_0$ such that
\eq{
\frac{E_k(\hat\ka)}{E_{k-1}(\hat\ka)}\geq C_3>\al,\quad E_i(\hat{\ka})\geq C_3^i,\quad \forall i=1,\ldots k-1.
}
\end{lemma}
\pf{
By \autoref{lemma-estimate-F},  $F\geq C_2$ on $[0,T)$, where $C_2>0$ depends only on $\cM_0$. Hence, for $i=1,\ldots, k$, we have
\eq{
C_3:=\al+C_2 \leq \frac{E_k(\hat\ka)}{E_{k-1}(\hat\ka)}\leq E_k(\hat\ka)^\frac{1}{k}\leq E_i(\hat\ka)^\frac{1}{i}.
}
}

\begin{lemma}\label{thm-C2-estimate}
There exists $C_4>0$ depending only on $\cM_0$ such that 
\eq{ \label{curvature-estimate}
\ka_i(\cdot,t) \leq C_4.
}
\end{lemma}
\begin{proof}
In view of \autoref{curvature est MP}, since $F$ is concave, we have 
\eq{
\frac{d}{dt}\max \frac{h_i^j}{u} &\leq \frac{u}{F} \bigg(-\frac{(h^2)_i^j}{u^2}\bigg)+2\frac{h_i^j}{u} \leq -C_0 \frac{(h^2)_i^j}{u^2}+2\frac{h_i^j}{u},
}
where we used $1-\frac{F^{kl}h_{kl}}{F}\geq 0$ and that $u$ is uniformly bounded below away from zero and $F$ is uniformly bounded above (cf. \autoref{prop-C1-estimate} and \autoref{lemma-estimate-F}). Therefore, $\frac{h_i^j}{u}$ is uniformly bounded above. Since $u\leq r=e^{\ga} \leq C$ due to \autoref{prop-C0-estimate}, $h_i^j$ is uniformly bounded above.
\end{proof}

Now we complete the a priori estimates for the flow \eqref{sum-hessian-flow}.
\begin{prop}\label{prop-longtime-existence}
    Under the assumption of \autoref{main-thm-flow}, the flow \eqref{sum-hessian-flow} exists for all times with uniform $C^{\ell}$ bounds for all $\ell\in \bbN$.
\end{prop}
\begin{proof}
    From \autoref{prop-C0-estimate}, \autoref{prop-C1-estimate}, we have uniform $C^0$ and $C^1$ estimates. In view of \autoref{thm-C2-estimate}, the principal curvatures are uniformly bounded above. Hence, in view of  \autoref{boundary of Ga-al-k-cone} and \autoref{preserving (a,k)-convexity}, the shifted principal curvatures $\hat{\ka}$ lie within a compact region (independent of time) of the convex set $\Ga_{\alpha,k}$. Therefore, the operator $\cL=\partial_t-\frac{1}{F^2}F^{ij}\nabla_i\nabla_j$ is uniformly parabolic. Moreover, since $F$ is concave, we can apply Krylov-Safonov's H\"{o}lder estimate and Caffarelli \cite[Thm. 3.1]{Caf89}  (or \cite[Thm. 8.1]{CC95}) to \eqref{general-flow-radial} to deduce the $C^{2,\beta}$ bounds (see \cite[Appendix]{BH17}). Finally,  the uniform $C^{\ell}$ bounds (for all $\ell\in \bbN$) follows from Schauder theory. Therefore, the flow exists for all positive times.
\end{proof}

\section{Approximation of  weakly \texorpdfstring{$(\al,2)$}{}-convex hypersurfaces}\label{sec: approximation}
In this section, we employ a mean curvature type flow to prove the following theorem.
\begin{thm}\label{approximation} 
Let $\cM$ be a star-shaped, weakly $(\alpha,2)$-convex hypersurface. Then we can smoothly approximate $\cM$ by a sequence of star-shaped and $(\alpha,2)$-convex hypersurfaces.
\end{thm}

Recall from the first few lines of \autoref{rem: conj} that $\cM$ is $(\al,2)$-convex if and only if its principal curvatures lie in the convex set
\eq{
\ti \Ga_{\alpha,2}=\Ga_1\cap \{\la\in \bbR^n:\, E_2(\la)+\al E_1(\la)>0\}.
}
Moreover, we have
\eq{\label{boundary of (alpha,2)-cone}
\partial \ti \Ga_{\alpha,2}= \{\la\in \bar{\Gamma}_1:\, E_2(\la)+\al E_1(\la)=0\}.
}

The proof of \autoref{approximation} will be completed in several steps. We consider a modified mean curvature flow in $\bbR^{n+1}$:
\eq{ \label{modified-mean-curvature-flow}
x&: \cN^n\times [0,T)\to \bbR^{n+1}\\
\dot{x}&=-(E_1+\beta)\nu,
}
where $\beta$ is a positive constant to be determined later. Write $\cN_t$ for the solution to this flow with $\cN_0$ a closed, embedded, smooth hypersurface in $\bbR^{n+1}$.

Along the flow \eqref{modified-mean-curvature-flow}, by \eqref{evol-Weingarten-matrix-general} and Simon's identity we have
\eq{ \label{evol-Weingarten-mMCF}
\partial_th_i^j &=\nabla^j\nabla_i (E_1+\beta)+(E_1+\beta)(h^2)_i^j \\
&=\frac{1}{n}\De h_i^j+\frac{|A|^2}{n}h_i^j-E_1 (h^2)_i^j+(E_1+\beta)(h^2)_i^j\\
&=\frac{1}{n}\De h_i^j+\frac{|A|^2}{n}h_i^j+\beta(h^2)_i^j.
}

Given a convex subset $C$ of $\bbR^n$, let $\op{SC}$ denote the set of supporting affine functionals for $C$; that is, the set of affine linear maps $\ell:\bbR^n \ra \bbR$ satisfying $|D\ell|=1$ and $\ell(z)\geq 0$ for all $z\in C$ with $\ell(z_0)=0$ at some point $z_0\in \partial C$ (cf. \cite{ACGL20}).

\begin{lemma}\cite[Lem. 9.3]{ACGL20}\label{lemma-cone-property}
Let $C$ be a convex subset of $\bbR^n$.  If $C$ is a cone in $\bbR^n$, then the supporting affine functionals are linear functionals.
\end{lemma}

Choose $\beta:=\frac{2}{n\al}\max_{\cN_0}|A|^2$. On a possibly very small time interval, denoted again by $[0,T)$, we have
\eq{
\max_{\cN_t}|A|^2 < 2\max_{\cN_0}|A|^2, \quad t\in [0,T).
}

\begin{lemma}\label{weakly (a,k)-convexity}
If $\cN_0$ is weakly $(\al,2)$-convex, then $\cN_t$ is weakly $(\al,2)$-convex for $t\in [0,T)$.
\end{lemma}

\begin{proof}
Let $Q:=\{\kappa\in \bbR^n: \sum_{i=1}^n\kappa_i^2\geq n\alpha \beta \}$.
On the time interval $[0,T)$, $A(\cdot,t)$ avoids the set $Q$. Let us define the vector
\eq{
\Phi(\kappa)=\(\frac{1}{n}\kappa_1\sum_{i=1}^n\kappa_i^2+\beta \kappa_1^2,\ldots, \frac{1}{n}\kappa_n\sum_{i=1}^n\kappa_i^2+\beta \kappa_n^2\).
}
Recall that  $\bar{\ti\Ga}_{\alpha,2}$ is a closed convex set and contains the positive cone. By the maximum principle \cite[Thm. 4]{CL04} (see also \cite[Thm. 4.2]{Ham86}), it suffices to show that every solution $\ka(t)=(\kappa_1(t),\ldots,\kappa_n(t))$ to the ODE
\eq{
\left\{\begin{aligned}\frac{d}{dt}\kappa(t)&=\Phi(\kappa(t)),\\
\kappa(0)&=\kappa \in \ov{\ti \Ga}_{\al,2}\setminus Q 
\end{aligned}
\right.
}
either remains in $\ov{\ti \Ga}_{\al,2}$ or else enters $Q$ at some time $t_{0}\in (0,T)$.

Suppose $\ka(t)$ never enters $Q$. We show that $\ka(t)$ remains in $\ov{\ti\Ga}_{\al,2}$. If $\kappa \in  \partial \ti\Ga_{\al,2}$, then, by \eqref{boundary of (alpha,2)-cone}, we must have $E_2+\al E_1=0$ at $\ka$. If $E_1=E_2=0$, then $\kappa \in \partial \Ga_{2}$. From \autoref{lemma-cone-property} we deduce that $\Phi(\kappa)$ is in the tangent cone to $\bar{\Ga}_{2}$ at $\kappa$. Therefore, due to \cite[Thm. 4.1]{Ham86}, $\kappa(t)$ remains in $\bar{\Ga}_{2}$ and hence in $\bar{\ti\Ga}_{\al,2}$. The only remaining case is when $E_2+\al E_1=0$ and $E_{1}>0$. 
In view of \autoref{lemma-new-Newton-ineq}, $(E_{2}+\alpha E_{1})^2\geq (E_{3}+\alpha E_2)(E_{1}+\alpha )$, and hence $E_{3}+\al E_{2}\leq 0$.
Now we compute
\eq{
\frac{d}{dt}(E_{2}+\al E_{1})&= -\frac{\al}{n}E_{1}\sum_{i=1}^n\kappa_i^2 -\beta (n-2)E_{3}+\al^2\beta (n-1)E_{1}\\
&\geq \al E_{1}\bigg(\al \beta-\frac{1}{n}\sum_{i=1}^n\kappa_i^2\bigg)\\
&>0.
}
Hence, $\kappa(t)$ cannot escape $\ov{\ti\Ga}_{\al,2}$.
\end{proof}

\begin{proof}[Proof of \autoref{approximation}]
Let $\cN_0=\cM$. We show that $\cN_t$ is $(\al,2)$-convex for all $t\in (0,T)$.
Due to \autoref{weakly (a,k)-convexity}, $F=\frac{E_2+\al E_{1}}{E_{1}+\al }\geq 0$ for $t\in [0,T)$ (in this proof, all $E_i$ are calculated at the principal curvatures of $\cN_t$). We calculate
\eq{
\partial_t F&=F^i_j\partial_t h_i^j \\
            &=F^i_j\(\frac{1}{n}\De h_i^j+\frac{|A|^2}{n}h_i^j+\beta (h^2)_i^j\)\\
            &=\frac{1}{n}\De F-\frac{1}{n}F^{ij,pq}\nabla_m h_{pq} \nabla^m h_{ij}+\frac{|A|^2}{n}F^{ij}h_{ij}+\beta F^{ij}(h^2)_{ij}.
}
Since $F$ is concave in $\Ga_{1}$, for $t\in [0,T)$:
\eq{
\partial_t F \geq \frac{1}{n}\De F+ \frac{|A|^2}{n}F^{ij}h_{ij}+\beta F^{ij}(h^2)_{ij}.
}
Using \autoref{lem-Em-property}, we calculate
\eq{
F^{ij}h_{ij}&=\sum_i \(\frac{E_2^{ii}+\al E_1^{ii}}{E_1+\al}-\frac{E_2+\al E_1}{(E_1+\al)^2}E_1^{ii}\)\ka_i\\
&=\frac{E_1E_2+2\al E_2+\al^2 E_1}{(E_1+\al)^2}\\
&=-\frac{\al E_{1}}{E_{1}+\al}+\(1+\frac{\al}{E_1+\al}\)F\\
&\geq -\frac{\al E_1}{E_1+\al}
}
and
\eq{
F^{ij}(h^2)_{ij}&=\sum_i \(\frac{E_2^{ii}+\al E_1^{ii}}{E_1+\al}-\frac{E_2+\al E_1}{(E_1+\al)^2}E_1^{ii}\)\ka_i^2\\
&=\frac{(n-1)E_{2}+n\al E_{1}}{E_{1}+\al }F -\frac{(n-2)E_{3}+\al(n-1)E_2}{E_{1}+\al }\\
&\geq -\frac{(n-2)E_{3}+\al(n-1)E_2}{E_{1}+\al }\\
 &\geq -(n-2)\frac{E_{3}+\alpha E_2}{E_{1}+\al }+\frac{\al^2  E_{1}}{E_{1}+\al }-\alpha F.
}
Hence, using $(E_{2}+\alpha E_{1})^2\geq (E_{3}+\alpha E_2)(E_{1}+\alpha )$ by \autoref{lemma-new-Newton-ineq}, we find
\eq{
\partial_t F&\geq  \frac{1}{n}\De F+\frac{\al E_1}{E_{1}+\al }\( -\frac{|A|^2}{n}+\al \beta \)-(n-2)\beta\frac{E_3+\alpha E_2}{E_1+\al }-\alpha \beta F\\
&\geq \frac{1}{n}\De F-(n-2)\beta F^2-\alpha \beta F.
}
By the strong maximum principle, $F$ is positive for all $t\in (0,T)$.
\end{proof}

\section{Convergence}
\begin{lemma}\label{prop-monotonicity}Let $\Omega_t$ denote the domain whose boundary is $\cM_t$.
Along the flow \eqref{general-flow}, the value of $\sum_{i=0}^{k-1}\beta_i\ti W_{k-i}(\Omega_t)$ is preserved while $\sum_{i=0}^{k-1}\beta_i\ti W_{k+1-i}(\Omega_t)$ is monotone decreasing.
\end{lemma}
\begin{proof}
Using the Minkowski integral identity \eqref{Minkowski-formula}, we compute
\eq{
  \frac{d}{dt} \sum_{i=0}^{k-1}\beta_i\ti W_{k-i}(\Omega_t)= \int_{\cM_t} \sum_{i=0}^{k-1}\beta_i E_{k-i}(\ka)\(\frac{\sum_{i=0}^{k-1}\beta_iE_{k-1-i}(\ka)}{\sum_{i=0}^{k-1}\beta_iE_{k-i}(\ka)}-u\) d\mu= 0.
} 
On the other hand,
\eq{\label{monotonicity x}
  &~\frac{d}{dt} \sum_{i=0}^{k-1}\beta_i\ti W_{k+1-i}(\Omega_t)\\
=&~\int_{\cM_t} \sum_{i=0}^{k-1}\beta_i E_{k+1-i}(\ka)\(\frac{\sum_{i=0}^{k-1}\beta_iE_{k-1-i}(\ka)}{\sum_{i=0}^{k-1}\beta_iE_{k-i}(\ka)}-u\) d\mu\\
=&~ \int_{\cM_t} \frac{\(\sum_{i=0}^{k-1}\beta_i E_{k+1-i}(\ka)\)\(\sum_{i=0}^{k-1}\beta_iE_{k-1-i}(\ka)\)-\(\sum_{i=0}^{k-1}\beta_iE_{k-i}(\ka)\)^2}{\sum_{i=0}^{k-1}\beta_iE_{k-i}(\ka)} d\mu\leq  0.
}
Here, we used the Minkowski integral identity, and the last inequality follows from the Newton-Maclaurin type inequality \eqref{new-Newton-type-ineq}.
\end{proof}

\begin{prop}\label{subsequence convergence}
    Along the flow \eqref{sum-hessian-flow}, we have 

    \eq{
    \limsup_{t\to\infty}\int_{\cM_t} \frac{\(\sum_{i=0}^{k-1}\beta_i E_{k+1-i}\)\(\sum_{i=0}^{k-1}\beta_iE_{k-1-i}\)-\(\sum_{i=0}^{k-1}\beta_iE_{k-i}\)^2}{\sum_{i=0}^{k-1}\beta_iE_{k-i}} d\mu=0.
    }

\end{prop}
\begin{proof}
If that is not the case, by \autoref{prop-monotonicity}, for some $\varepsilon>0$ and $T_{\varepsilon}$, we have
\eq{
\int_{\cM_t} \frac{\(\sum_{i=0}^{k-1}\beta_i E_{k+1-i}\)\(\sum_{i=0}^{k-1}\beta_iE_{k-1-i}\)-\(\sum_{i=0}^{k-1}\beta_iE_{k-i}\)^2}{\sum_{i=0}^{k-1}\beta_iE_{k-i}} d\mu\leq -\varepsilon,~\forall t>T_{\varepsilon}.
}
Therefore, by \eqref{monotonicity x}, we have
\eq{
\int_{T_{\varepsilon}}^{s} \frac{d}{dt}\(\sum_{i=0}^{k-1}\beta_i\ti W_{k+1-i}(\Omega_t)\)dt\leq -\varepsilon (s-T_{\varepsilon}).
}
Sending $s\to\infty$ (cf. \autoref{prop-longtime-existence}), we obtain 
\eq{
\lim_{t\to \infty} \(\sum_{i=0}^{k-1}\beta_i\ti W_{k+1-i}(\Omega_t)\)=-\infty,
}
which is a contradiction.
\end{proof}

\begin{proof}[Proof of \autoref{main-thm-flow}]
In view of the uniform $C^{\ell}$ estimates (for all $\ell \in \bbN$), \autoref{subsequence convergence} and \autoref{lemma-new-Newton-ineq}, for a subsequence of times $\{t_i\}$,
$\M_{t_i}$ converges smoothly to a sphere. By \autoref{prop-monotonicity}, the radius $r_\infty$ is uniquely determined by 
    \eq{
    r_\infty:=g_{\alpha,k}^{-1}\(\sum_{i=0}^{k-1}\beta_i\ti W_{k-i}(\Omega_0)\).
    }
 Using \autoref{prop-monotonicity} again, we have
\eq{
\lim_{t\to \infty} &\(\sum_{i=0}^{k-1}\beta_i\ti W_{k+1-i}(\Omega_t)\)=\sum_{i=0}^{k-1}\beta_i\ti W_{k+1-i}(B_{r_{\infty}}).
}
Hence, by \autoref{lemma-new-Newton-ineq}, all limit hypersurfaces are spheres of radius $r_{\infty}$.

Next, we show that all limit hypersurfaces are centred at their origin. In view of  \autoref{prop-C0-estimate}, 
\eq{
\frac{d}{dt}\max_{\bbS^n} \ga(\cdot,t)\leq 0 \quad \text{and} \quad \frac{d}{dt}\min_{\bbS^n} \ga(\cdot, t)\geq 0.
}

Let $A:=\lim_{t\to\infty} \max_{\bbS^n} \ga $ and $B:=\lim_{t\to\infty} \min_{\bbS^n} \ga $. Suppose a limiting hypersurface is not centred at the origin, then
\eq{
e^B< r_\infty < e^A.
}
Since $F(r_\infty^{-1})=r_{\infty}^{-1}$, the Lipchitz functions $\max_{\bbS^n} \ga(\cdot,t)$ and $\min_{\bbS^n} \ga(\cdot,t)$, at their differentiability points (provided $t$ is large enough), would satisfy
\eq{
\frac{d}{dt}\max_{\bbS^n} \ga<\frac{1}{2}\(\frac{r_{\infty}}{e^A} -1\)\quad \text{and} \quad \frac{d}{dt}\min_{\bbS^n} \ga>\frac{1}{2}\(\frac{r_{\infty}}{e^B} -1\).
}
\end{proof}
 
\begin{proof}[Proof of \autoref{main-thm-ineq}]~

(i) Let $\Omega_0$ be a bounded domain with smooth boundary $\cM_0=\partial\Omega_0$, such that $\cM_0$ is star-shaped and $(\alpha,k)$-convex. In view of \autoref{prop-monotonicity} and \autoref{main-thm-flow}, we deduce the desired inequality \eqref{new-quermassintegral-ineq} as follows: 
    \eq{
    \sum_{i=0}^{k-1}\beta_i\ti W_{k+1-i}(\Omega_0)
    &\geq \sum_{i=0}^{k-1}\beta_i\ti W_{k+1-i}(B_{r_\infty})\\
    &=f_{\al,k}\bigg(\sum_{i=0}^{k-1}\beta_i\ti W_{k-i}(B_{r_\infty})\bigg)\\
    &=f_{\al,k}\bigg(\sum_{i=0}^{k-1}\beta_i\ti W_{k-i}(\Omega_0)\bigg).
    }
    Moreover, if equality holds, then the following equality holds on $\cM_0$:
      \eq{
\bigg(\sum_{i=0}^{k-1}\beta_i E_{k+1-i}(\ka)\bigg)\bigg(\sum_{i=0}^{k-1}\beta_iE_{k-1-i}(\ka)\bigg)=\bigg(\sum_{i=0}^{k-1}\beta_iE_{k-i}(\ka)\bigg)^2.
    }
    Since $\hat{\ka}\in\Ga_{\alpha,k}$, owing to \autoref{lemma-new-Newton-ineq} we conclude that $\cW=a\mathrm{I}$ for some constant $a>0$, and hence $\Omega_0$ is a ball.  

    (ii) In view of  \autoref{approximation}, the inequality \eqref{new-quermassintegral-ineq} also holds for star-shaped, weakly $(\alpha,2)$-convex hypersurfaces.
    
    Suppose that $\cM$ is a star-shaped, weakly $(\alpha,k)$-convex hypersurface such that the equality in \eqref{new-quermassintegral-ineq} holds. Let 
    \eq{
    \cM_{+}:=\{ x\in \cM :\hat{\ka}(x)\in\Ga_{\alpha,k} \}.
    }
  It is clear that $\cM_{+}$ is open. Since $\cM$ is compact and embedded, there exists an elliptic point in $\cM$, so $\cM_{+}$ is nonempty. We will employ an argument of Guan-Li \cite{GL09} to show that $\cM_+=\cM$:
    
    We claim that $\cM_{+}$ is closed. Choose $\eta\in C^2(\cM)$ with compact support in $\cM_+$. Define the variation $y_s=x+s\eta(x) \nu(x)$, where $x$ is the position vector of $\cM$ and $\nu$ is the unit outward normal of $\cM$ at $x$. Let $\Omega_s$ be the domain enclosed by $\cM_s$ whose position vector is $y_s$. For sufficiently small $s$, $\cM_s$ is star-shaped and weakly $(\alpha,k)$-convex, and the equality holds for $\Omega_0=\Omega$. Thus, we have

\eq{
0&=\frac{d}{ds}\Big|_{s=0}\left[ \sum_{i=0}^{k-1}\beta_i\ti W_{k+1-i}(\Omega_s)-f_{\al,k}\bigg( \sum_{i=0}^{k-1}\beta_i\ti W_{k-i}(\Omega_s)\bigg)\right]\\
&=\int_{\cM} \bigg(\sum_{i=0}^{k-1}\beta_i E_{k+1-i}(\ka)-c\sum_{i=0}^{k-1}\beta_i E_{k-i}(\ka)\bigg)\eta d\mu
}
for some constant $c>0$. Since $\eta$ is arbitrary, on $\cM_+$ we have
    \eq{
   \sum_{i=0}^{k-1}\beta_i E_{k+1-i}(\ka)&=c\sum_{i=0}^{k-1}\beta_i E_{k-i}(\ka)\\
   &=c\(\frac{E_k(\hat{\ka})}{E_{k-1}(\hat{\ka})}-\al\)E_{k-1}(\hat{\ka})>0. 
    }
    On the other hand, by \eqref{new-Newton-type-ineq}, we have
    \eq{
   \bigg(\sum_{i=0}^{k-1}\beta_i E_{k-i}(\ka)\bigg)^2\geq \bigg(\sum_{i=0}^{k-1}\beta_iE_{k+1-i}(\ka)\bigg)  \bigg(\sum_{i=0}^{k-1}\beta_iE_{k-1-i}(\ka)\bigg),
    }
    which implies that
    \eq{
 \sum_{i=0}^{k-1}\beta_i E_{k-i}(\ka)\geq c \sum_{i=0}^{k-1}\beta_iE_{k-1-i}(\ka) =c E_{k-1}(\hat{\ka}).
    }
That is,
\eq{
\frac{E_k(\hat{\ka})}{E_{k-1}(\hat{\ka})}-\al \geq c.
}
Hence,
\eq{
E_{k-1}^{\frac{1}{k-1}}(\hat{\ka})\geq \al+c,\q E_i^{\frac{1}{i}}(\hat{\ka})>\al+c\q \forall i=1,\ldots,k-2.
}
Therefore, $\cM_+$ is closed and $\cM_+=\cM$. By the previous case, $\cM$ is a sphere.
\end{proof}

\section{A Poincar\'{e} type inequality}\label{sec: poincare-type inequalities}
In this section, we prove the following theorem.
\begin{thm-}
Let $1< k\leq n$ and $\cM$ be a $k$-convex hypersurface. Then for all $f\in C^2(\cM)$ we have
\eq{\label{k-inequality}
\int \frac{(\operatorname{tr}_{\dot{\sigma}_k}(\operatorname{Hess}f))^2}{\si_k}d\mu \geq k\int \langle \cW(\nabla f),\nabla f\rangle_{\dot{\sigma}_k}d\mu.
}
\end{thm-}
\begin{lemma} \label{lemma for poincare} Let $1<k\leq n$
and $\cM$ be a closed, embedded, smooth hypersurface.
For every $f,\psi\in C^2(\cM)$, we have
\eq{
\int \si_k^{ij,pq}f_{;ij}\psi_{;pq} d\mu= (k-1)\int\si_k^{im}h_m^j\partial_i \psi\partial_j f d\mu.
}
\end{lemma}
\begin{proof}
The tubular neighborhood of an embedded hypersurface $\cM\subset \bbR^{n+1}$ can be defined by
\eq{
\Phi: \cM \times (-\varepsilon,\varepsilon) &\ra U_{\varepsilon}\subset \bbR^{n+1},\\
      (x,s) \quad &\mapsto x+s\nu(x),
}
where $\nu(x)$ is the unit outward normal at the point $x\in \cM$. Here, $\varepsilon>0$ is a small constant, which is guaranteed by the fact that the map is a one-to-one correspondence. Thus, for any $y\in U_\varepsilon$, there exists a unique point $x\in \cM$ such that $y=x+s\nu(x)$.  The function $f$ on $\cM$ can be extended to $U_{\varepsilon}$ as $\bar f(y)=f(x)$. By the Gauss-Weingarten formula, we have
\eq{
\operatorname{Hess}^{\bbR^{n+1}}_{ij} \bar f &=\operatorname{Hess}_{ij} f+\frac{\partial \bar f}{\partial \nu}h_{ij}=f_{;ij}.
}

Now, along the normal variation $\dot{x}|_{t=0}=-\psi \nu$, at time $t=0$ we have
\eq{
\partial_t g_{ij} &=-2\psi h_{ij},\\
\partial_t g^{ij} &=2\psi h^{ij},\\
\partial_t h_i^j &= \psi_{;i}{}^j+\psi (h^2)_i^j,\\
\partial_t d\mu &=-\psi\cH  d\mu,
}
where $\cH$ denotes the mean curvature of $\cM$. 
Moreover, we have
\eq{
\partial_t \si_k^{ij}=\partial_t (g^{im}(\si_k)_m^j)
=&~g^{im}\frac{\partial^2 \si_k}{\partial h_j^m \partial h_r^s} \partial_t h_r^s+(\partial_t g^{im}) (\si_k)_m^j\\
=&~\si_k^{ij,pq}(\psi_{;pq}+\psi(h^2)_{pq})+2\psi \si_k^{jm}h_m^i.
}
Therefore,
\eq{
0= \frac{d}{dt}\Big|_{t=0}\int_{\cM_t} \si_k^{ij}\bar f_{;ij} d\mu=&~\int_{\cM} \si_k^{ij,pq}f_{;ij} \psi_{;pq}+\psi\si_k^{ij,pq}(h^2)_{pq}f_{;ij} d\mu\\
&~+\int_{\cM} 2\psi \si_k^{jm}h_m^i f_{;ij}-\si_k^{ij}(\partial_t \Ga_{ij}^m) \partial_m f-\psi\cH \si_k^{ij}f_{;ij} d\mu.
}
In an orthonormal basis that diagonalizes $h_{ij}$, using \cite[(4.17)]{GM03}, we calculate
\eq{
 \psi \si_k^{ij,pq}(h^2)_{pq}f_{;ij}=&~\psi \sum_{i=1}^n\sum_{p=1,p\neq i}^{n}\frac{\partial \si_{k-1}(\ka|i)}{\partial \ka_p}\ka_p^2f_{;ii}\\
=&~ \psi \sum_{i=1}^n (\si_1(\ka|i)\si_{k-1}(\ka|i)-k\si_k(\ka|i))f_{;ii} \\
=&~ \psi (\si_2^{im}(\si_k)_m^i-k\si_{k+1}^{ij})f_{;ij} \\
=&~ \psi (\cH \si_k^{ij}-(\si_k)^{jm}h_m^{i}-k\si_{k+1}^{ij})f_{;ij},
}
where we used $\si_2^{ij}=\cH g^{ij}-h^{ij}$ by the definition of the Newton operator
\eq{ \label{Newton-operator}
\si_{k+1}^{ij}=\si_k g^{ij}-\si_k^{jm}h_m^i, \quad 1\leq k<n,
} 
and $\si_1^{ij}=g^{ij}$, and we also take $\si_{n+1}^{ij}=0$ by convention.
We also have
\eq{
\int_{\cM} -\si_k^{ij}(\partial_t \Ga_{ij}^m) \partial_mf  d\mu&= \int_{\cM} \si_k^{ij}(2\partial_i \psi h_j^m- \partial_l \psi g^{ml} h_{ij}+\psi h_{ij;}{}^m)\partial_m f  d\mu\\
&=\int_{\cM} 2\si_k^{ij}h_j^m \partial_i \psi \partial_m f -k\si_k \langle \nabla \psi,\nabla f\rangle+\psi \langle \nabla \si_k,\nabla f\rangle d\mu\\
&=\int_{\cM} 2\si_k^{ij}h_j^m \partial_i \psi \partial_m f-(k+1)\si_k \langle \nabla \psi,\nabla f\rangle-\si_k \psi \De f d\mu.
}
Hence,
\eq{
  0
=&~\int_{\cM} \si_k^{ij,pq}f_{;ij}\psi_{;pq}+ (2\si_k^{im}h_m^j-(k+1)\si_k g^{ij})\partial_i \psi \partial_j f d\mu\\
 &~+\int_{\cM} \psi (-k\si_{k+1}^{ij}+\si_k^{jm}h_m^i-\si_k g^{ij})f_{;ij} d\mu\\
=&~\int_{\cM} \si_k^{ij,pq}f_{;ij}\psi_{;pq} -(k-1) \si_k^{im}h_m^j\partial_i \psi \partial_j f  d\mu,
}
where we used \eqref{Newton-operator} in the last equality.
\end{proof}
\begin{rem}\label{k=2, Bochner formula}
If we take $k=2$ and $f=\psi$ in \autoref{lemma for poincare}, then we recover the Bochner formula:
\eq{
\int (\De f)^2-|\operatorname{Hess}f|^2d\mu=\int \si_2^{im}h_m^j\partial_if \partial_jfd\mu =\int \operatorname{Ric}(\nabla f,\nabla f)d\mu.
}
\end{rem}

\begin{proof}[Proof of \autoref{thm: k-inequality}]
 By the concavity of $\si_k^{1/k}$ in $\Ga_k$, we have
\eq{
\si_k^{ij,pq} f_{;ij}f_{;pq}\leq \frac{k-1}{k}\frac{(\si_k^{ij}f_{;ij})^2}{\si_k}.
}
Now take $\psi=f$ in \autoref{lemma for poincare}.
\end{proof}

\begin{rem}
Let $f=c+\langle x, w\rangle$ where $w\in \bbR^{n+1}$ is a fixed vector, $c$ is a constant and $x$ is the position vector of $\cM$. Then
\eq{
\si_k^{ij,pq} f_{;ij}f_{;pq}=\frac{k-1}{k}\frac{(\si_k^{ij}f_{;ij})^2}{\si_k}.
}
\end{rem}

\begin{cor}Let $1< k\leq n$ and $\cM$ be a convex hypersurface. Then
\eq{
\int \frac{|\nabla \sigma_k|^2}{\sigma_k}d\mu\geq \int \sum_{i<j} \frac{\si_k^{ii}\si_k^{jj}}{\si_k}\ka_i\ka_j(\ka_i-\ka_j)^2d\mu.
}
\end{cor}
\begin{proof}
Note that
\eq{
\nu_{;i}&=h_i^m e_m\\
\nu_{;ij}&=(\nabla^m h_{ij})e_m-(h^2)_{ij}\nu.
}
Let $\{v_{\ell}\}$ be an orthonormal basis of $\bbR^{n+1}$.
Applying \eqref{k-inequality} to each  $f_{\ell}=\langle \nu,v_{\ell}\rangle$ for $\ell=1,\ldots, n+1$ and then summing over $\ell$ we obtain
	\eq{\label{stability}
	\int \frac{|\nabla \sigma_k|^2+(\sum \sigma_k^{ii}\kappa_i^2)^2}{\sigma_k}d\mu\geq k\int \sum \sigma_k^{ii}\kappa_i^3d\mu.
	}
Now recall that for any constants $a_i,b_j\in \bbR$, there holds
\eq{
0 \leq  \frac{1}{2}\sum_{i,j=1}^{n}(a_ib_j-a_jb_i)^2
   =  \sum_i a_i^2 \sum_j b_j^2 -\big(\sum_{i} a_ib_i \big)^2.
}
Let $a_i= \sqrt{\si_k^{ii}}\ka_i^{\frac{3}{2}}$ and $b_j=\sqrt{\si_k^{jj}}\ka_j^\frac{1}{2}$. Then we have
\eq{
\sum_{i} \si_k^{ii}\ka_i^3 \sum_{i} \si_k^{ii}\ka_i-\big(\sum_{i} \si_k^{ii}\ka_i^2\big)^2=&~\frac{1}{2}\sum_{i,j=1}^n \si_k^{ii}\si_k^{jj} (\ka_i^\frac{3}{2}\ka_j^{\frac{1}{2}}-\ka_j^\frac{3}{2}\ka_i^{\frac{1}{2}})^2 \\
=&~\frac{1}{2}\sum_{i,j=1}^n \si_k^{ii}\si_k^{jj} \ka_i\ka_j(\ka_i-\ka_j)^2 \\
=&~\sum_{i<j} \si_k^{ii}\si_k^{jj} \ka_i\ka_j(\ka_i-\ka_j)^2.
}
Hence, for $1<k\leq n$:
\eq{
\int \frac{|\nabla \sigma_k|^2}{\sigma_k}d\mu\geq \int \sum_{i<j} \frac{\si_k^{ii}\si_k^{jj}}{\si_k}\ka_i\ka_j(\ka_i-\ka_j)^2d\mu.
}
\end{proof}
\section*{Acknowledgment}
The first author's work was supported by the National Key Research and Development Program of China 2021YFA1001800, the National Natural Science Foundation of China 12101027, and the Fundamental Research Funds for the Central Universities. Both authors were supported by the Austrian Science Fund (FWF) under Project P36545. 

\providecommand{\bysame}{\leavevmode\hbox to3em{\hrulefill}\thinspace}

\vspace{10mm}
\textsc{School of Mathematical Sciences, Beihang University,\\ Beijing 100191, China}
\email{\href{mailto:huyingxiang@buaa.edu.cn}{huyingxiang@buaa.edu.cn}}
	
	\vspace{3mm}
	\textsc{Institut f\"{u}r Diskrete Mathematik und Geometrie,\\ Technische Universit\"{a}t Wien, Wiedner Hauptstra\ss e 8-10,\\ 1040 Wien, Austria,}
	\email{\href{mailto:yingxiang.hu@tuwien.ac.at}{yingxiang.hu@tuwien.ac.at}}
	
	\vspace{3mm}
\textsc{Institut f\"{u}r Diskrete Mathematik und Geometrie,\\ Technische Universit\"{a}t Wien, Wiedner Hauptstra{\ss}e 8-10,\\ 1040 Wien, Austria,} \email{\href{mailto:mohammad.ivaki@tuwien.ac.at}{mohammad.ivaki@tuwien.ac.at}}

\end{document}